\documentclass[10pt]{amsart}

\usepackage{amssymb,latexsym,amsmath}
\usepackage{graphics}                 
\usepackage{color}                    
\usepackage{hyperref}                 
\usepackage[all]{xy}

\numberwithin{equation}{section}
\theoremstyle{plain} 
\newtheorem{theorem}[equation]{Theorem}
\newtheorem{corollary}[equation]{Corollary}
\newtheorem{lemma}[equation]{Lemma}
\newtheorem{proposition}[equation]{Proposition}

\theoremstyle{definition}

\newtheorem{example}[equation]{Example}
\newtheorem{conjecture}[equation]{Conjecture}

\theoremstyle{remark}
\newtheorem{remark}[equation]{Remark}
\newtheorem{question}[equation]{Question}

\def\C{{\mathbb{C}}}

\def\P{{\mathbb{P}}}
\def\Q{{\mathbb{Q}}}

\def\Z{{\mathbb{Z}}}
\def\Ch{{\rm Ch}}

\def\sp{{\rm SP}}
\def\supp{{\rm supp}}

\def\cZ{{\mathcal{Z}}}

\title{The multiplicative group action on singular varieties and Chow varieties}
\author{Wenchuan Hu}


\address{
School of Mathematics\\
Sichuan University\\
Chengdu 610064\\
P. R. China
}

\email{huwenchuan@gmail.com}

\begin{document}

\begin{abstract}
We answer two questions of Carrell on a singular complex projective variety admitting the multiplicative group action,  one positively and the other negatively. The results are applied to Chow varieties and we obtain  Chow groups of 0-cycles and Lawson homology groups of 1-cycles for Chow varieties. A short survey on the structure of the Chow varieties is included for comparison and completeness. Moreover, we  give counterexamples to Shafarevich's question on the rationality of the irreducible components of Chow varieties.
\end{abstract}

\maketitle
\pagestyle{myheadings}
 \markright{Holomorphic vector field and Chow groups}

\tableofcontents

\section{Introduction}


Let $V$ be a holomorphic vector field defined on a  projective algebraic variety
$X$. The zero subscheme  $Z$ is the subspace of $X$ defined by the ideal generated by $V\mathcal{O}_X$ and we denote it by $X^{V}$.

The existence of a holomorphic vector field with zeroes on a
smooth projective variety imposes restrictions on the topology  of the manifold. For examples, the Hodge numbers
$h^{p,q}(X)=0$ if $|p-q|>\dim Z$ (see \cite{Carrell-Lieberman}). For a smooth complex projective variety $X$ admitting a $\C^*$-action, Bialynicki-Birula structure theorem describes
the relation between the structure of $X$ and that of the fixed points set(\cite{Bialynicki-Birula2}). In \cite{Hu0}, one get the identified the algebraic geometric invariants
such as the Chow group, Lawson homology with the corresponding singular homology with rational coefficients in the case that $X^V$ is of zero dimensional.

According to Lieberman (\cite{Lieberman1}), a holomorphic vector field $V$ on a complex algebraic projective variety $X$ with nonempty zeroes is equivalent to
the 1-parameter group $G$ generated by $V$ is a product of $\C^*$'s and at most one $\C$'s. This induces us to study the structure of $X$ admitting a
$\C^*$-action or a $\C$-action.

In this paper, we consider the case that $X$ is a singular projective variety admitting a $\C$-action (resp. $\C^*$-action). In these case, the relation between the structure of $X$ and that of the fixed point set
is subtle.  A general result of Bialynicki-Birula says that $X$ and the fixed point set are $\C$-equivalent(resp. $\C^*$-equivalent).
In \cite{Hu0}, we got the Chow group of zero cycles and Lawson homology group of 1-cycles for $X$ admitting a $\C$-action.

When $X$ is singular and admitting a $\C^*$-action, the Bialynicki-Birula type structure theorem also holds for singular homology groups if the action is  ``good" in sense of \cite{Carrell-Goresky}.
In general, it does not hold for a singular $X$ admitting a $\C^*$-action without additional conditions.
For a singular variety $X$ admitting the certain  $\C^*$-action with isolated fixed points, Carrell asked if  the odd Betti numbers  of $X$ vanish,etc.
In section \ref{section4}, we ask parallel questions to those of Carrell and gives  answers to all of them. We give counterexamples to  some of these questions.
We compute the Chow groups of 0-cycles for singular varieties admitting a $\C^*$-action with isolated fixed points. As a contrast to projective varieties admitting a $\C$-action,
 the parallel result for Lawson homology group of 1-cycles does not holds any more (see Example \ref{Exam4.26}).

In section \ref{Section5}, we briefly review and summarize some known  algebraic and topological invariants for Chow varieties $C_{p,d}(\P^n)$
 parameterizing effective $p$-cycles of degree $d$ in the complex projective space $\P^n$.  We  give a counterexample to the question of  Shafarevich on the rationality of the irreducible components of Chow varieties,
 based on the work of Eisenbud, Harris, Mumford, etc.

 As applications of section \ref{section2} and \ref{section4}, we
 compute the chow group of zero cycles and Lawson homology groups of  1-cycles for Chow varieties.

\section{Invariants under the additive group action}
\label{section2}
Let $X$ be a possible singular complex
projective algebraic variety $X$ admitting an additive group action.
Our main purpose is to compare certain algebraic and topological invariants (such as the Chow group of zero cycles, Lawson homology, singular homology, etc.) of $X$ to those of the fixed point set $X^{\C}$.
If $X$ is smooth projective, most of topological invariants are studied and  computed in details, but some of algebraic invariants are still hard to identified.
 Some of those invariants have been investigated even if $X$ is singular. In this section, we will identify some of these invariants including the Chow groups of zero cycles, Lawson homology
 for 1-cycles, singular homology with integer coefficients, etc.

\subsection{A-equivalence}

Let $A$ be a fixed  complex quasi-projective algebraic variety. Recall that  an algebraic scheme $X_1$ is \textbf{simply
$A$-equivalent} to an algebraic variety $X_2$ if $X_1$ is isomorphic to a closed subvariety $X_2'$ of
$X_2$ and there exists an isomorphism $f: X_2- X_2'\to Y \times A$, where $Y$ is an algebraic variety.
The smallest equivalence relation containing the relation of simple $A$-equivalence is called
the \textbf{A-equivalence} and we denote it by $\sim$ (see \cite{Bialynicki-Birula1}). A result of Bialynicki-Birula
says that $X{\sim} X^{\C}$ if $X$ is a quasi-projective variety admitting a $\C$-action.
A similar statement  holds for  $X$ admitting $\C^*$-action.
From this, Bialynicki-Birula showed that  $H^0(X,\Z)\cong H^0(X^{\C},\Z)$ and $H^1(X,\Z)\cong H^1(X^{\C},\Z)$ in the case that $X$ admits a $\C$-action, where
$\chi(X)=\chi(X^{\C^*})$ in the case that $X$ admits a $\C$-action (see \cite{Bialynicki-Birula1}). Along this route, more additive invariants has been
calculated for varieties admits a $\C$ or $\C^*$-action (see \cite{Hu}).

\subsection{Chow Groups and Lawson homology}

Let $X$ be any complex projective variety or scheme of dimension $n$ and let $\cZ_p(X)$  be the group of algebraic $p$-cycles on $X$.
 Let $\Ch_p(X)$ be the Chow group of $p$-cycles on $X$, i.e.
$\Ch_p(X)=\cZ_p(X)/{\rm \{rational ~equivalence\}}$.  Set $\Ch_p(X)_{\Q}:=\Ch_p(X)\otimes \Q$, $\Ch_p(X)=\bigoplus_{p\geq0} \Ch_p(X)$.
For more details on Chow theory, the reader is  referred to Fulton (\cite{Fulton}).

\begin{proposition}\cite{Hu0}\label{Prop2.2}
Let $X$ be a (possible singular) connected  complex projective variety. If $X$ admits a $\C$-action with isolated fixed points, then $\Ch_0(X)\cong \Z$.
\end{proposition}

\begin{remark}
More generally, by using the same method, we can show that if $X$ admits a $\C$-action with  fixed points $X^{\C}$, then $\Ch_0(X)\cong \Ch_0(X^{\C})$.
\end{remark}

The \emph{Lawson homology}
$L_pH_k(X)$ of $p$-cycles for a projective variety is defined by
$$L_pH_k(X) := \pi_{k-2p}({\mathcal Z}_p(X)) \quad for\quad k\geq 2p\geq 0,$$
where ${\mathcal Z}_p(X)$ is provided with a natural topology (cf.
\cite{Friedlander1}, \cite{Lawson1}).

In \cite{Friedlander-Mazur}, Friedlander and Mazur showed that there
are  natural transformations, called \emph{Friedlander-Mazur cycle class maps}
\begin{equation}\label{eq01}
\Phi_{p,k}:L_pH_{k}(X)\rightarrow H_{k}(X,\Z)
\end{equation}
for all $k\geq 2p\geq 0$.

Set
{$$
\begin{array}{llcl}
&L_pH_{k}(X)_{hom}&:=&{\rm ker}\{\Phi_{p,k}:L_pH_{k}(X)\rightarrow
H_{k}(X)\};\\
&L_pH_{k}(X)_{\Q}&:=&L_pH_{k}(X)\otimes\Q.
\end{array}
 $$}

Denoted by $ \Phi_{p,k,\Q}$ the map $ \Phi_{p,k}\otimes{\Q}:L_pH_{k}(X)_{\Q}\rightarrow H_{k}(X,\Q) $.
The \emph{Griffiths group} of dimension $p$-cycles is defined to be
$$
{\rm Griff}_p(X):={\mathcal Z}_p(X)_{hom}/{\mathcal Z}_p(X)_{alg}.$$

Set
$$
\begin{array}{lcl}
{\rm Griff}_p(X)_{\Q}&:=&{\rm Griff}_p(X)\otimes\Q;\\
\end{array}
$$

It was proved by Friedlander \cite{Friedlander1} that, for any
smooth projective variety $X$, $$L_pH_{2p}(X)\cong {\mathcal
Z}_p(X)/{\mathcal Z}_p(X)_{alg}.$$

Therefore
\begin{eqnarray*}
L_pH_{2p}(X)_{hom}\cong {\rm Griff}_p(X).
\end{eqnarray*}

\begin{proposition}\cite{Hu0}\label{Prop2.4}
Under the same assumption as Proposition \ref{Prop2.2},  we have
$$L_1H_k(X)\cong H_k(X,\Z)$$ for all $k\geq 2$. In particular, ${\rm Griff}_1(X)=0$.
\end{proposition}
\begin{remark}
The isomorphism $L_0H_k(X)\cong H_k(X,\Z)$ holds for any integer $k\geq 0$, which is the special case of the Dold-Thom Theorem.
\end{remark}

\begin{remark}
The assumption of ``connectedness" in Proposition \ref{Prop2.4} is not necessary. By the same reason, we can remove the connectedness in Proposition \ref{Prop2.2}, while the conclusion
``$\Ch_0(X)\cong \Z$" would
be replaced by ··$\Ch_0(X)\cong H_0(X,\Z)$”.
\end{remark}

\subsection{The virtual Betti  and Hodge numbers }
 Recall that
 the \emph{virtual Hodge polynomial} $H:Var_{\C}\to \Z[u,v]$ is defined by  the following properties:
\begin{enumerate}
 \item $H_X(u,v):=\sum_{p,q}(-1)^{p+q}\dim H^{q}(X,\Omega_X^p)u^pv^q$ if $X$ is nonsingular and  projective (or complete).
\item  $H_X(u,v)=H_U(u,v)+H_Y(u,v)$ if $Y$ is a closed algebraic subset of $X$ and $U=X-Y$.
\item  If $X=Y\times Z$, then $H_X(u,v)=H_Y(u,v)\cdot H_Z(u,v)$.
\end{enumerate}

The existence and uniqueness of such a polynomial follow from Deligne's Mixed Hodge theory (see \cite{Deligne1,Deligne2}).
The coefficient of $u^pv^q$ of $H_X(u,v)$ is called the \emph{virtual Hodge $(p,q)$-number} of $X$ and we denote it by $\tilde{h}^{p,q}(X)$.
Note that from the definition, $\tilde{h}^{p,q}(X)$ coincides with the usual Hodge number $(p,q)$-number ${h}^{p,q}(X)$ if $X$ is a smooth
projective variety. The sum $\tilde{\beta}^k(X):=\sum_{i+j=k}\tilde{h}^{p,q}(X)$ is called the $k$-th \emph{virtual Betti number} of $X$.
The \emph{ virtual Poincar\'{e} polynomial} of $X$ is defined to be
$$\widetilde{P}_X(t):=\sum_{k=0}^{2\dim_{\C} X} \beta^k(X)t^k,$$
which coincides to the usual Poincar\'{e} polynomial defined through the corresponding usual Betti numbers.

\section{Results related to the multiplicative group action}
\label{section4}
In this section  we will give all kinds of  relations between a complex variety (not necessarily smooth, irreducible) and
the fixed point set of  a multiplicative group action or an additive group action.

Let $X$ be a smooth complex projective variety which admits a $\C^*$-action with fixed point set $X^{\C^*}$. Denote by
 $F_1,\cdots,F_r$ the connected components. It was shown by Bialynicki-Birula that there is a homology basis formula (\cite{Bialynicki-Birula2}):
\begin{equation}\label{equation1}
H_k(X,\Z)\cong \bigoplus_{j=1}^r H_{k-2\lambda_j}(F_j,\Z),
\end{equation}
where $\lambda_j$ is the fiber dimension of the bundle  in $P_j:X_j^+\to F_j$ and $X_j^+:=\{x\in X:\lim_{t\to 0}{t\cdot x\in F_j}\}$.
This result has been generalized to compact K\"{a}hler manifolds  without change by Carrell-Sommese \cite{Carrell-Sommese2} and Fujiki independently.
In fact, when $X$ is a compact K\"{a}hler manifold,  the Hodge structure on $X$ is completely determined by those on the fixed point
set in an obvious way.

Furthermore, there are similar basis formulas for Chow groups (see \cite{Chow} for $X^{\C^*}$ finite and  \cite{Karpenko} for the general case) and Lawson homology (see
see \cite{Lima-Filho} for $X^{\C^*}$ finite and\cite{Hu-Li} for the general case), as  applications of
Bialynicki-Birula' structure theorem (\cite{Bialynicki-Birula2}).

However, if $X$ is a singular projective algebraic variety, Equation (\ref{equation1}) would be failed in general.
Under some additional condition, Equation (\ref{equation1}) may still hold.
For example,
if the $\C^*$-action on $X$ is ``good" in the sense of Carrell and Goresky,
Equation (\ref{equation1}) has been shown to hold (cf. \cite{Carrell-Goresky}).

There are several questions related to the structure of $X$ and $X^{\C^*}$.
J. Carrell asked the question  how does the mixed Hodge structure on $X$ relate to the mixed Hodge structure on the fixed point set in the case of
good action.
\begin{question}(\cite[p.21]{Carrell}) \label{Ques4.0}
 In the case of a good action, how does the mixed Hodge structure on $X$ relate to the mixed Hodge structure on $X^{\C^*}$ ?
\end{question}

We will give an explicit relation on the mixed Hodge structure between $X$ and $X^{\C^*}$, especially the relation of their
virtual Hodge numbers (see Proposition \ref{Prop4.9}).

When $X$ is a possibly singular complex projective variety  with a $\C^*$-action, where a ``variety" means a reduced, not necessary irreducible scheme,
Carrell and Goresky showed that there still exists an integral homology basis formula under the assumption that the $\C^*$-action is ``good" (\cite{Carrell-Goresky}).

 Carrell asked the following question.

\begin{question}(\cite[p.22]{Carrell})\label{Ques4.1}
If an  irreducible complex projective variety $X$ admits not necessarily good  $\C^*$-action with isolated fixed points, do the odd homology groups of $X$ vanish?
\end{question}

The following example gives a negative answer to his question.

\begin{example}\label{exam4.2}
\emph{
Let $C$ be a cubic plane curve with a node singular point $p$, e.g. $(zy^2=x^3+x^2z)\subset \P^2$. The normalization $\sigma:\tilde{C}\to C$
of $C$ is isomorphic to $\P^1$. Let $\C^*\times \P^1\to \P^1$ be the holomorphic $\C^*$-action given by $(t, [x:y])\mapsto [t x:y]$.
The fixed point set of this action contains  two points, $[1:0]$ and $[0:1]$.  We can always assume  $\sigma([1,0])=\sigma([0:1])=p_0$ by composing
a suitable automorphism of $\P^1$, where $p_0=[0:0:1]$ is the singular point of $C$.
The holomorphic $\C^*$-action on $\P^1$ descends to a holomorphic $\C^*$-action on $C$ whose fixed point set is the single point $p$.  More explicitly, such a map $\sigma$ can be given by the formula:
$\sigma:\P^1\to C$, $[s:t]\mapsto [st(s+t):st(s-t):(s+t)^3]$.}

\emph{However, the fundamental group of $C$ is isomorphic to $\Z$, so $H_1(C,\Z)\cong \Z$ and $\beta_1(C)=1\neq0$.}
\end{example}

In each dimension $n\geq 1$, there exists a projective variety $X$ satisfying the assumption in Question \ref{Ques4.1} such
that $\beta_1(X)\neq 0$. To see this, note that $\P^{n-1}$ admits a $\C^*$-action with isolated fixed points for each integer $n\geq 1$. Hence
$X:=C\times \P^{n-1}$  admits an induced $\C^*$-action from each component with isolated fixed points.
Therefore, we get   $\beta_1(X)=\beta_1(C)$  by the K\"{u}nneth formula and the later is nonzero from Example \ref{exam4.2}.

In Example \ref{exam4.2}, $X$ admits a $\C^*$-action but the odd homology group $H_1(X,\Z)$ is nonzero. However, the odd
\emph{virtual Betti numbers}  and the \emph{virtual Hodge numbers}  $\tilde{h}^{p,q}(X)$ are zero, where $p\neq q$.
To see this, we can write $C=\C^*\cap p_0$ and so  $H_{C}(u,v)=(uv-1)+1=uv$. Hence $\tilde{{h}}^{1,0}(C)=\tilde{{h}}^{0,1}(C)=0$ and $\tilde{\beta}^1(C)=0$.

In certain sense, the virtual Betti numbers are more suitable to reveal the topology of a singular variety.
A natural question would be the following modified version of Carrell's Question in virtual Betti numbers.

\begin{question}[Carrell]\label{Ques4.2}
If an  irreducible complex  projective variety $X$ admits not necessarily good  $\C^*$-action with isolated fixed points, do the odd virtual Betti numbers of $X$ vanish?
\end{question}

If $X$ is irreducible and $\dim X=1$, the answer to the question is positive.
 In this case, $X=\C^*\cup Y$ and $Y$ is a set of finite points.
Then  $H_{X}(u,v)=(uv-1)+k=uv+k-1$ and the odd  virtual Betti numbers of $X$ are zero,
where $k$ is the number of points of $Y$.

If $X$ is smooth projective, then the answer to the question is positive (\cite{Bialynicki-Birula2}).  Moreover, if the $\C^*$-action on $X$ is
``good" in the sense of Carrell and Goresky, the answer is also positive (see Corollary \ref{Cor4.13} for a weaker condition such that the answer is positive).

The following example of a projective variety admits a  not ``good" $\C^*$-action, but the answer to Question \ref{Ques4.2} is positive.
\begin{example}
Let $X:=\sp^d(\P^n)$ be the $d$-th symmetric product of the complex projective space $\P^n$.
 The standard $(\C^*)^n$-action on $\P^n$ induces a
$(\C^*)^n$-action on $\sp^d(\P^n)$ with isolated fixed points.
 It follows from Cheah \cite{Cheah} that the $k$-th virtual Betti number of $\sp^d(\P^n)$  is the coefficient of  $t^dx^k$ in the
 power series of $\prod_{j=0}^n(1-tx^{2j})^{-1}$.  Hence $\tilde{\beta}_k(\sp^d(\P^n))=0$ for and all $d$ and all odd $k$.
\end{example}
Under a weaker condition than Carrell and Goresky's ``good" condition, the answer to Question \ref{Ques4.2} is positive (see Corollary \ref{Cor4.13}).

However, in general, the answer to Question \ref{Ques4.2} is negative.
There is an irreducible projective algebraic surface $S$ admitting $\C^*$-action with isolated zeroes such that  the first virtual betti number $\tilde{\beta}_1(S)\neq 0$.
Such a surface was constructed by Lieberman (\cite[p.111]{Lieberman2}) as a nonrational surface admitting a holomorphic vector field with isolated zeroes. A suitable modification
fulfills our purpose.
The following example gives a negative answer to Question \ref{Ques4.2}.

\begin{example}\label{exam4.7}
 Let $Y=\P^1\times C$, where $C$ is a smooth projective curve with genus $g(C)\geq 1$. Let us consider the $\C^*$-action $\phi:\C^*\times Y\to Y$  given by
$(t, ([u:v], z))\to ([u:tv],z)$, where $[u:v]$ denotes the homogeneous coordinates for $\P^1$ and $z$ denotes the coordinate for the curve $C$. The fixed point of the
action $\phi$ are $C_1:=[1:0]\times C$ and $C_2:=[0:1]\times C$. These curves has self-intersection zero. Let $ {\sigma}:\widetilde{S}\to Y$ be  obtained from $Y$ by blowing up one point $p_i$
on each $C_i$ ($i=1,2$), and let $\tilde{\phi}:\C\times \widetilde{S}\to \widetilde{S}$ be the equivariant lifting action. The fixed point of $\tilde{\phi}$ are the proper
transforms $\widetilde{C}_i$ of $C_i$ and two other isolated points. Since the self-intersection number of $\widetilde{C}_i$ on $\widetilde{S}$ is $-1$. One can blow down $\tilde{\sigma}:\widetilde{S}\to S$ the
$\widetilde{C}_i$ to obtain a projective surface $S$, which admitting the induced $\C^*$-action. Moreover $S^{\C^*}$ are four isolated points.  In explicitly,
 we have  the following relations
 \begin{equation}\label{equ4.8}
 \xymatrix{ \widetilde{S}\ar[r]^{\tilde{\sigma}}\ar[d]^{\sigma}& S\\
 Y\ar@{=}[r]& \P^1\times C.
 }
 \end{equation}

Now we can compute the virtual Betti numbers from the construction. Since $\widetilde{S}-\widetilde{C}_1-\widetilde{C}_1\cong S-\tilde{\sigma}(\widetilde{C}_1)-\tilde{\sigma}(\widetilde{C}_2)$
and $\widetilde{S}-\P^1-\P^1\cong Y-p_1-p_2$, we have by using the additive property of the virtual Poincar\'{e} polynomial
$$
\begin{array}{ccl}
\widetilde{P}_S(t)&=&\widetilde{P}_{\widetilde{S}}(t)-\widetilde{P}_{\widetilde{C}_1}(t)-\widetilde{P}_{\widetilde{C}_2}(t)+
    \widetilde{P}_{\tilde{\sigma}(\widetilde{C}_1)}(t)+\widetilde{P}_{\tilde{\sigma}(\widetilde{C}_2)}(t)\\
&=& \widetilde{P}_{\widetilde{S}}(t)-2\widetilde{P}_C(t)+2\\
&=& \widetilde{P}_{Y}(t)+2\widetilde{P}_{\P^1}(t)-2 -2\widetilde{P}_C(t)+2\\
&=& \widetilde{P}_{\P^1\times C}(t)+2\widetilde{P}_{\P^1}(t)-2 -2\widetilde{P}_C(t)+2\\
&=& \widetilde{P}_{\P^1}(t)\widetilde{P}_C(t)+2\widetilde{P}_{\P^1}(t)-2 -2\widetilde{P}_C(t)+2\\
&=& (t^2+1)(t^2+2g(C)t+1)+2(t^2+1)-2(t^2+2g(C)t+1)\\
&=& t^4+2g(C)t^3+2t^2-2g(C)t+1.
\end{array}
$$
Since $g(C)\geq 1$, $\tilde{\beta}_1(S)=-2g(C)\neq 0$.
\end{example}

\begin{remark}
We can also  construct examples of projective varieties in any dimension greater than or equals to 2 such that the answer to Question \ref{Ques4.2} is negative.
Since $\P^n$ admits a $\C^*$-action with isolated fixed points, so $S\times \P^n$ admits a  $\C^*$-action with isolated points, where $S$ is the projective surface constructed
in Example \ref{exam4.7}. By using
the product property of the virtual Poincar\'{e} polynomial, it is easy to
compute that $\tilde{\beta}_1(S\times\P^n)=-2g(C)$.
\end{remark}

Now we shall show that  the answer to   Question \ref{Ques4.2} is positive under certain not ``good" condition.
For a singular variety $X$ with a $\C^*$-action, one can always find an analytic Whitney stratification  whose strata are $\C^*$-invariant.
Recall that  the $\C^*$-action on $X$ is \emph{singularity preserving} as $t\to 0$ if there exists an equivariant Whiteny stratification of $X$ such that
for every stratum $A$, and for every $x\in A$, the limit $x_0=\lim_{t\to0}t\cdot x$ is also in $A$ (cf. \cite{Carrell-Goresky}). In this case,
$X=\bigcup_{j=1}^r X^+_j$, and $X^+_j\to F_j$ is a topologically locally trivial affine space bundle (cf. \cite[Lemma 1]{Carrell-Goresky}).
Denote $m_j$ be the dimension of the fiber of the bundle $X^+_j\to F_j$.

Then we have the following relation on virtual Hodge polynomials between $X$ and the fixed point set.

\begin{proposition}\label{Prop4.9}
Suppose $X$ admits a Whitney stratification which is singularity preserving as $t\to 0$. Then
$$
H_X(u,v)=\sum_{j=1}^r H_{F_j}(u,v)u^{m_j}v^{m_j},
$$
where  $F_j$ and $m_j$ are given as before.
\end{proposition}
\begin{proof}
Suppose  $X$ has a Whitney stratification that is singularity preserving as $t\to 0$ and let $F_j$ denote a fixed point component.
For a stratum $A$, the map $F_j\cap A$,  $p_j^{-1}(F_j\cap A):=\{x\in X:\lim_{t\to 0}(t\cdot x)\in F_j\cap A\}$ is
 Zariski locally trivial affine space bundle (cf. \cite{Bialynicki-Birula2}, \cite{Carrell-Sommese1}).
Since $F_j=\cup_{A\in \mathbb{S}} (F_j\cap A)$, where $\mathbb{S}$ is the set of all strata of $X$ in the given Whitney stratification.
Hence the total space of the topological locally trivial affine space bundle $X_j^+\to F_j$ can be written the disjoint union of subvarieties
$p_j^{-1}(F_j\cap A)$.

  Therefore, we have
 $$
 \begin{array}{ccl}
 H_{X}(u,v)&=&\sum_{j=1}^rH_{X_j^+}(u,v)\\
 &=&\sum_{j=1}^r \sum_{A\in \mathbb{S}} H_{p_j^{-1}(F_j\cap A)}(u,v)\\
 &=&\sum_{j=1}^r \sum_{A\in \mathbb{S}} H_{F_j\cap A}(u,v)\cdot H_{\C^{m_j}}(u,v)\\
&=&\sum_{j=1}^rH_{F_j}(u,v)\cdot H_{\C^{m_j}}(u,v)\\
&=&\sum_{j=1}^rH_{F_j}(u,v)(uv)^{m_j}.
 \end{array}
 $$
\end{proof}

\begin{remark}
Proposition \ref{Prop4.9} does not have to hold if the  singularity preserving property fails. For example, $X$ is the cone in $\P^{n+1}$ over a smooth
projective variety $V\subset \P^n=(z_{n+1}=0)$ with vertex $\P^0=[0:\cdots:0:1]$, the $\C^*$-action on $X$ induced by the action $(t,[z_0:\cdots:z_n:z_{n+1}])\mapsto [tz_0:\cdots:tz_n:z_{n+1}]$ on $\P^{n+1}$. The fixed point set is $V$ and $\P^0$, and
the action is not singularity preserving as $t\to 0$. In this case we observe that $H_X(u,v)\neq H_V(u,v)+h_{\P^0}(u,v)u^nv^n$.
However, if the action is given as $(t,[z_0:\cdots:z_n:z_{n+1}])\mapsto [z_0:\cdots:z_n:tz_{n+1}]$ on $\P^{n+1}$, it is
singularity preserving as $t\to 0$. So $H_X(u,v)=H_V(u,v)uv+H_{\P^0}(u,v)=H_V(u,v)uv+1$.
\end{remark}

From the proof of the above theorem, we see that if $X$ can be decomposed as the disjoint union of locally closed subvarieties
 (not necessarily irreducible) $W_j$ for $j=1,\cdots, r$, where $W_i$ is a locally trivial affine space bundle over $F_j$ with fiber $\C^{m_j}$ in
 Zariski topology, then $H_X(u,v)=\sum_{j=1}^r H_{F_j}(u,v)u^{m_j}v^{m_j}$.

From Proposition \ref{Prop4.9}, we see that the mixed Hodge structure of $X$ is partial determined by the mixed Hodge structures of the fixed point set.
One also obtains from Proposition \ref{Prop4.9} that  the  virtual Hodge numbers  of $X$ is nonnegative if all $F_j$ are smooth projective varieties.

\begin{corollary}
Suppose $X$ admits a Whitney stratification which is singularity preserving as $t\to 0$. Then
$$
\tilde{h}^{p,q}(X)=0, \quad \forall |p-q|>\dim X^{\C^*}.
$$
In particular, if $ X^{\C^*}$ contains only isolated points, then $\tilde{h}^{p,q}(X)=0$ for all $p\neq q$.
\end{corollary}

One obtains the relations between virtual Betti numbers of $X$ and those of the fixed point set immediately from Proposition \ref{Prop4.9}.
\begin{corollary}\label{Cor4.10}
Suppose $X$ admits a Whitney stratification which is singularity preserving as $t\to 0$. Then
\begin{equation}\label{equation2}
\tilde{P}_X(t)=\sum_{j=1}^r \tilde{P}_{F_j}(t)t^{2m_j}.
\end{equation}
\end{corollary}

If the $\C^*$-action on a projective variety $X$ is ``good" in the sense of Carrell and Goresky (cf. \cite{Carrell-Goresky}), then the
usual Poincar\'{e} polynomial $P_X(t)$ of $X$ can be expressed in terms of that of the fixed point set  as follows:
\begin{equation}\label{equation3}
{P}_X(t)=\sum_{j=1}^r {P}_{F_j}(t)t^{2m_j}.
\end{equation}
Furthermore, if $F_j$ are smooth projective varieties, then $\tilde{P}_X(t)={P}_X(t)$ since $\tilde{P}_{F_j}(t)={P}_{F_j}(t)$ for each $F_j$ and
Equation (\ref{equation2})-(\ref{equation3}). In other words, the virtual Betti numbers and the usual Betti numbers coincide for such projective
 varieties.  This gives us the following corollary.

Since the answer to Question \ref{Ques4.2} is  negative in general, the following corollary gives a sufficient condition for the $\C^*$-action such that  the odd virtual Betti numbers
vanish. This condition is much weaker than   Carrell and Goresky's   ``good" condition.

\begin{corollary}\label{Cor4.13} Under the assumption in Question \ref{Ques4.2} and suppose $X$ admits a Whitney stratification which is singularity preserving as $t\to 0$. Then
$$
\tilde{\beta}_{2k-1}(X)=0, \quad \forall k>\dim X^{\C^*}.
$$
In particular, if $ X^{\C^*}$ contains only isolated points, then $\tilde{\beta}_{k}(X)=0$ for all odd $k$.
\end{corollary}

  For a $\C^*$-action on algebraic varieties,  there is  a relation between
 virtual Hodge numbers  between $X$ and  $X^{\C^*}$ (see \cite{Hu}), i.e.,
\begin{equation}\label{equation4}
 \sum_{p-q=i} \tilde{h}^{p,q}(X)=\sum_{p-q=i} \tilde{h}^{p,q}(X^{\C^*}), \forall i.
\end{equation}

 If we set $\tilde{b}_{even}(X):=\sum_{i}\tilde{b}_{2i}(X)$ and $\tilde{b}_{odd}(X):=\sum_{i}\tilde{b}_{2i-1}(X)$, then
we get from equation \eqref{equation4}
\begin{equation}\label{equation4.15}
\begin{array}{ccl}
\tilde{b}_{even}(X)&=&\tilde{b}_{even}(X^{\C^*})\\
 \tilde{b}_{odd}(X)&=&\tilde{b}_{odd}(X^{\C^*}).
\end{array}
\end{equation}
 In particular, $X$ admits a $\C^*$-action with isolated zeroes, then $ \tilde{b}_{odd}(X)=0$, i.e., the sum of  all odd virtual
  Betti numbers  is zero.

  Note that the Euler characteristic $\chi(X)$ of $X$ is equal to $\tilde{b}_{even}(X)-\tilde{b}_{odd}(X)$
  and Equation \eqref{equation4.15} implies the fixed point formula for the Euler characteristic: $\chi(X)=\chi(X^{\C^*})$.

When $X$ admits $\C$-action with isolated fixed point, it was shown that $\Ch_0(X)\cong\Z$ (see Proposition \ref{Prop2.2}). Inspired by this result,
it is natural to ask if  $\Ch_0(X)\cong\Z$ holds for a $\C^*$-action. Amazingly, such a statement still holds.

\begin{proposition}\label{Prop4.19}
If $X$ is a connected projective variety admitting a $\C^*$-action with isolated fixed point, then we have $\Ch_0(X)\cong\Z$.
\end{proposition}
\begin{proof}
Since  $X$  admits a $\C^*$-action with isolated fixed points, there exists a $\C^*$-invariant Zariski open set $U\subset X$  such that $U\cong U'\times \C^*$ (see \cite{Bialynicki-Birula1}).
Such  $U$ and $U'$  can be assumed to be non-singular if necessary. Set $Z=X-U$. By the localization sequence of higher chow groups and homotopy invariance, we get
$\Ch_0(U'\times \C^*,1)\cong \Ch_0(U')$. From the Poicar\'{e} duality, homotopy invariance of cohomology and the K\"{u}nneth formula for the Borel-Moore homology, we obtain that
$H_1^{BM}(U'\times \C^*)\cong H^{2n-1}(U'\times \C^*)\cong H^{2n-1}(U'\times S^1)\cong H_0^{BM}(U'\times S^1)\cong H_0^{BM}(U')$. Note that the
cycle class map $ \Ch_0(U')\to H_0^{BM}(U',\Z)$ is always surjective. Hence the higher cycle class map $\phi_0(U,1):\Ch_0(U,1)\to H_{1}^{BM}(U,\Z)$ is surjective.

By applying the localization sequence to $(X,Z)$ and using the natural transform for the higher chow group to the singular homology group, we get
{\small
\begin{equation}
\xymatrix{\Ch_0(U,1)\ar[r]\ar@{>>}[d]&\Ch_0(Z)\ar[r]\ar[d]^{\cong}&\Ch_0(X)\ar[r]\ar[d]&\Ch_0(U)\ar[r]\ar[d]^{\cong}&
0
\\
H_{1}^{BM}(U,\Z)\ar[r]&H_{0}(Z,\Z)\ar[r]&H_{0}(X,\Z)\ar[r]&H_{0}^{BM}(U,\Z)\ar[r]&0.
}
\end{equation}
}

By induction hypothesis, we have the isomorphism $\Ch_0(Z)\stackrel{\cong}{\to} H_{0}(Z,\Z)$. Note that
$\Ch_0(U)\cong \Ch_0(U'\times\C)=0$ since a point moving a $\C$ direction to infinite, which is not on $U$.
Therefore $\Ch_0(U)=0=H_{0}^{BM}(U,\Z)$. Now we get the isomorphism $\Ch_0(X)\stackrel{\cong}{\to} H_{0}(X,\Z)$ by the Five Lemma.
Hence  $\Ch_0(X)\stackrel{\cong}{\to} \Z$ since $X$ is connected.
This completes the proof of the proposition.
\end{proof}

\begin{remark}\label{Rem4.21}
In fact, from the proof of Proposition \ref{Prop4.19},  we have shown the following result:
 If $X$ is a connected projective variety admitting a $\C^*$-action with nonempty fixed point set $X^{\C^*}$, then the inclusion $i:X^{\C^*}\to X$ induces a surjective $\Ch_0(X^{\C^*})\to \Ch_0(X)$.
\end{remark}

\begin{remark}
If $X$ is smooth projective variety admitting a $\C^*$-action with isolated fixed point,
 then $X$ admits a cellular decomposition(see \cite{Bialynicki-Birula2}) and $\Ch_p(X)\cong H_{2p}(X,\Z)$
for all $p\geq 0$. However, in the case that $X$ is singular, $\Ch_p(X)\cong H_{2p}(X,\Z)$ can be wrong for $p>0$ by the following example.
\end{remark}

\begin{example}
Let $S$ be the surface construction in Example \ref{exam4.7}, $\Ch_1(S)\ncong H_{2}(S,\Z)$. Moreover, $\Ch_1(S)_{hom}\neq 0$.
Recall that the relations among $\widetilde{S}, S $ and $Y$ were given in diagram \eqref{equ4.8}.
By using $\tilde{\sigma}:\widetilde{S}\to S$ and the localization sequence for Chow group of 1-cycles, we
get the difference between $\Ch_1(\widetilde{S})$ and $\Ch_1(S)$ is at most rank 2 (generated by the cycle classes of $\widetilde{C}_1$ and $\widetilde{C}_2$) since
the sequence $\Ch_1(\widetilde{C}_1\cup \widetilde{C}_2)\to \Ch_1(\widetilde{S})\to \Ch_1({S})\to 0$ is exact and  $\Ch_1(\widetilde{C}_1\cup \widetilde{C}_2)\cong \Z\oplus\Z$.
So $\Ch_1(\widetilde{S})_{hom}\cong \Ch_1(S)_{hom}$.
On the other hand,   $\Ch_1(\widetilde{S})\cong \Ch_1(Y)\oplus \Z^2\cong (\Ch_0(C)\oplus\Z)\oplus\Z^2$. Hence $\Ch_1(\widetilde{S})_{hom}\cong \Ch_0(C)\cong J(C)\neq 0$, where
$J(C)$ is the Jacobi of $C$ of genus $g(C)\geq 1$.
\end{example}


By applying to a possible singular projective variety carrying a holomorphic vector field with isolated zeroes, we have the following result.
\begin{corollary}
Let $X$ be a (possible singular) complex projective algebraic variety which admits a holomorphic
vector field $V$ whose zero set $Z$ is isolated and nonempty. Then the cycle class map
 we have $\Ch_0(X)\cong \Z$.
\end{corollary}
\begin{proof}
Recall that a holomorphic vector field generates a $G$-action on $X$, where $G\cong (\C^*)^k\times \C$ or $G\cong (\C^*)^k$. Write $G\cong G_1\times \C^* $
and $X_1:=X^{\C^*}$. From Remark   \ref{Rem4.21}, the inclusion $X_1\to X$ induces a surjection $\Ch_0(X_1)\to \Ch_0(X)$. If  $G\cong (\C^*)^k$,  we get the surjection $\Ch_0(V)\to \Ch_0(X)$ by induction. If $G\cong (\C^*)^k\times \C$, we get the surjection  $\Ch_0(V_1)\to \Ch_0(X)$ by induction, where $V_1:=X^{(\C^*)^k}$. Note that $V_1$ admits a $\C$-action whose fixed point is $V$. By Proposition \ref{Prop2.2}, we have $\Ch_0(V)\cong \Ch_0(V_1)$.
Therefore, the inclusion $V\hookrightarrow X$ induces a surjection $\Ch_0(V)\to \Ch_0(X)$. By assumption, $V$ is finite points. Hence $\Ch_0(X)$ is of finite rank and so $\Ch_0(X)\to H_0(X,\Z)\cong \Z$ is injective. Clearly,  $\Ch_0(X)\neq 0$ and we get $\Ch_0(X)\cong \Z$.
\end{proof}


Applying to Lawson homology, we get structure for 1-cycles.
\begin{lemma} \label{Lemma4.23}
For for any projective variety $X$ and any integer $k\geq 2r\geq0$ and $n\neq0$, we have the following formula
\begin{equation}\label{eqn4.24}
L_rH_k(X\times\C^*)\cong L_{r-1}H_{k-2}(X)\oplus L_rH_{k-1}(X).
\end{equation}
\end{lemma}
\begin{proof}
 First, we  note that the pair $(X\times \C, X\times \{0\})$, we have
the long exact sequence of Lawson homology:
\begin{equation}\label{eqn4.25}
...\stackrel{\partial}{\to}L_rH_{k}(X)\stackrel{i_{*}}{\to} L_rH_k(X\times \C)\stackrel{Res}{\longrightarrow} L_rH_k(X\times \C^*)\stackrel{\partial}{\to} L_rH_{k-1}(X)\to...
\end{equation}
where $i:X=X\times\{0\}\to X\times\C$ is the inclusion, $Res$ is restriction map and $\partial$ is the boundary map.

The long exact sequence of Lawson homology for the pair $(X\times \P^1, X\times \{0\})$ is
$$
...\stackrel{\partial}{\to}L_rH_{k}(X)\stackrel{i_{\infty*}}{\to} L_rH_k(X\times \P^1)\stackrel{Res}
{\longrightarrow} L_rH_k(X\times \C)\stackrel{\partial}{\to} L_rH_{k-1}(X)\to...
$$ where $i_{\infty}:X=X\times \{\infty\}\to X\times\P^1$ is the inclusion.

Then, from the $\C^1$-homotopy invariance of Lawson homology, we get $i_{0*}=i_{\infty*}:L_pH_k(X)\to L_pH_k(X\times\P^1)$,
where $i_0:X=X\times \{0\}\to X\times\P^1$ is the inclusion. From the definition of $i$ and $i_0$, we have
$i_*=Res\circ i_{0*}$, where $Res:L_rH_k(X\times \P^1)\to L_rH_k(X\times \C)$ is the restriction map. Hence we obtain
\begin{equation*}
i_*=Res\circ i_{0*}=Res\circ i_{\infty*}=0.
\end{equation*}
Therefore, Equation (\ref{eqn4.25}) is broken into short exact sequences
$$
0{\to} L_rH_k(X\times \C)\stackrel{Res}{\longrightarrow} L_rH_k(X\times \C^*)\stackrel{\partial}{\to} L_rH_{k-1}(X)\to 0.
$$
This sequence splits since  the map $\cZ_r(X\times \C^*)=\cZ_r(X\times \C)/\cZ_r(X\times \{0\})\to
\cZ_{r-1}(X)\simeq \cZ_r(X\times \C)$ given by $c\mapsto c\cap (X\times\{0\})$ gives a section of the projection
$\cZ_r(X\times \C)\to \cZ_r(X\times \C)/\cZ_r(X\times \{0\})$. So we get Equation (\ref{eqn4.24}).
This completes the proof of the lemma.
\end{proof}

Now we study the structure of Lawson homology under a $\C^*$-action. When $X$ admits $\C$-action with isolated fixed point, it was shown that $L_1H_k(X)\cong H_k(X,\Z)$ (see Proposition \ref{Prop2.4}). Inspired by this result,
it is natural to ask the following question.

\begin{question}\label{ques4.28}
 Let $X$ be a complex projective variety admitting a $\C^*$-action with isolated fixed point. Does $L_1H_k(X)\cong H_k(X,\Z)$  hold for $k\geq 2$?
 \end{question}

The positive answer to this question would be   an analogue of Proposition \ref{Prop2.4}.
Contrary to the analogue between Proposition  \ref{Prop2.2} and \ref{Prop4.19}, it is  surprising to a certain degree that the answer to Question \ref{ques4.28} is negative
in the sense that for each $k\geq 2$, we can find $X$ (depending on $k$) satisfying conditions in the question such that $L_1H_k(X)\ncong H_k(X,\Z)$.

\begin{example}\label{Exam4.26}
Let $S$ be the variety given in Example \ref{exam4.7},  then $S\times S$ admits a $\C^*$-action with isolated fixed points induced by the $\C^*$-action on $S$.
We have
$$L_1H_2(S\times S)\cong H_2(S\times S,\Z),$$
and
$$ L_1H_3(S\times S)\ncong H_3(S\times S,\Z).$$
\end{example}
\begin{proof}
The $\C^*$-action $\phi:\C^*\times S\to S$, $(t,x)\mapsto\phi(t,x)$ induces a $\C^*$-action $(t, (x,y))\mapsto (tx, ty)$ on $S\times S$. The fixed point set $(S\times S)^{\C^*}\subset S^{\C^*}\times S^{\C^*} $ is finite since
$S^{\C^*}$ is.

By construction, we have $H_1(S,\Z)=0$. By K\"{u}nneth formula, $H_2(S\times S,\Z)\cong  H_2(S,\Z)\oplus H_2(S,\Z)$. Note $H_2(S,\Z)\cong \Z^3$ is generated by the homological classes of algebraic cycles
$\tilde{\sigma}(\sigma^{-1}(\P^1\times c_0)$, $\tilde{\sigma}(\sigma^{-1}(p_i))$, where  $c_0$ is a point of $C$ different from $p_i$ for $i=1,2$. Hence  $H_2(S\times S,\Z)$ is generated by algebraic cycles and
so the cycle class map  $L_1H_2(S\times S)\to H_2(S\times S,\Z)$ is surjective.

From the construction in Example \ref{exam4.7}, ${\sigma}:\widetilde{S}\to Y=C\times \P^1$ is the blow up of two point $p_1\in C_1, p_2\in C_2$. Set
 $U:=Y-C_1-C_2\cong \widetilde{S}-{\sigma}^{-1}(C_1)-{\sigma}^{-1}(C_2)$, where ${\sigma}^{-1}(C_i)=\widetilde{C}_i\cup E_i$ and $E_i\cong\P^1$. One gets $U\cong C\times \C^*$.
 Since $\tilde{\sigma}:\widetilde{S}\to S$ is the blow down and each $\widetilde{C}_i$ collapses to a point, $S- \tilde{\sigma}(E_1)-\tilde{\sigma}(E_2)\cong U$.
Since only $\widetilde{C}_i$ collapses under $\tilde{\sigma}$, $\tilde{\sigma}(E_i)\cong E_i\cong\P^1$. Set $Z:= S\times S-U\times U$ and $\widetilde{E}_i:=\tilde{\sigma}(E_i)$, then $Z$ is the union
$\big((\widetilde{E}_1\cup \widetilde{E}_2)\times S \big)\bigcup  \big(S\times (\widetilde{E}_1\cup \widetilde{E}_2)\big) $.   Set $\widetilde{Z}:=\widetilde{S}\times\widetilde{ S}-U\times U$ and then
$\widetilde{Z}$ is the union $\big(({\sigma}^{-1}(C_1)\cup {\sigma}^{-1}(C_2))\times \widetilde{S} \big)\bigcup  \big(\widetilde{S}\times ({\sigma}^{-1}(C_1)\cup {\sigma}^{-1}(C_2))\big) $.
From the long localization exact sequence of Lawson homology for $(\widetilde{S},\widetilde{Z})$ and $(S,Z)$, we have
{\tiny
$$
\xymatrix{ ...\ar[r]&L_1H_3(\widetilde{U})\ar[r]\ar[d]^{=} &L_1H_2(\widetilde{Z})\ar[d]\ar[r] &L_1H_{2}(\widetilde{S}\times \widetilde{S})\ar[d]^{(\sigma\times\sigma)_*}\ar[r]&L_1H_2(U)\ar[d]\\
...\ar[r] &L_1H_3(\widetilde{U})\ar[r] &L_1H_2(Z)\ar[r] &L_1H_{2}(S\times S)\ar[r] &L_1H_2(U)
 }
$$
}

By homotopy invariance  and localization sequences of Lawson homology, one gets
$L_1H_k(Z)\cong H_k(Z,\Z)$ and $L_1H_k(\widetilde{Z})\cong H_k(\widetilde{Z},\Z)$ for $k\geq 2$.
From construction, the collapse $\widetilde{Z}\to Z$ induces a surjective map $H_2(\widetilde{Z},\Z)\to H_2(Z,\Z)$.

From $U\cong C\times \C^*$ and Lemma 4.23, we get isomorphisms
 $$
 \begin{array}{ccl}
 L_1H_2(U\times U)
 &\cong& L_1H_2(C\times C\times \C^*\times\C^*)\\
 &\cong& L_0H_0(C\times C\times \C^*)\\
 &\cong& H_0^{BM}(C\times C\times \C^*,\Z)\\
 &=&0.
 \end{array}
 $$

Therefore,
$(\sigma\times\sigma)_*$ is a surjective map. Note that $\widetilde{S}\times \widetilde{S}$ is nonsingular and projective,
a directed computation by localization and blowup formula for Lawson homology (see \cite{Hu2}) yields $L_1H_2(\widetilde{S}\times \widetilde{S})_{hom}=0$.
Hence $L_1H_2(S\times S)_{hom}=0$ and $L_1H_2(S\times S)\to H_2(S\times S,\Z)$ is injective.

We need to identify  $L_1H_3(U\times U)$ and $H_3^{BM}(U\times U,\Z)$ so that one can compare that relation between
$L_1H_3(S\times S)$ and $H_3(S\times S,\Z)$.

 By Lemma \ref{Lemma4.23}, we get
 $$
 \begin{array}{ccl}
 L_1H_3(U\times U)
 &\cong& L_1H_3(C\times C\times \C^*\times\C^*)\\
  &\cong& L_0H_1(C\times C\times \C^*)\oplus L_1H_2(C\times C\times \C^*)\\
  &\cong& L_0H_0(C\times C)\oplus L_0H_0(C\times C)\\
 &\cong&  \Z\oplus\Z.
 \end{array}
 $$

It is not hard to check that
 $$
 \begin{array}{ccl}
 H_3^{BM}(U\times U,\Z)
 &\cong& H_3^{BM}(C\times C\times \C^*\times\C^*,\Z)\\
  &\cong& H_1^{BM}(C\times C\times \C^*,\Z)\oplus H_2^{BM}(C\times C\times \C^*,\Z)\\
  &\cong& H_0^{BM}(C\times C)\oplus H_0^{BM}(C\times C)\oplus  H_1^{BM}(C\times C)\\
 &\cong&  \Z\oplus\Z\oplus H_1(C\times C).
 \end{array}
 $$

Hence the cycle class map
$$\Phi_{1,3}(U\times U):L_1H_3(U\times U)\to H_3^{BM}(U\times U,\Z)$$
is  not surjective. In particular, $\Phi_{1,3}(U)$ is not an isomorphism.

For simplicity in diagram $X:=S\times S$, $\widetilde{U}:=U\times U$. From the following commutative diagram
{\tiny
$$
\xymatrix{L_1H_{3}(Z)\ar[r]\ar[d]^{\cong}& L_1H_3(X)\ar[r]\ar[d]^{\cong?} &L_1H_3(\widetilde{U})\ar[r]\ar[d]^{\Phi_{1,3}(\widetilde{U})} &L_1H_2(Z)\ar[d]^{\cong}\ar[r] &L_1H_{2}(X)\ar[d]^{\cong}\\
H_{3}(Z,\Z) \ar[r]& H_{3}(X,\Z)\ar[r] &H_{3}^{BM}(\widetilde{U},\Z)\ar[r] &H_{2}(Z,\Z)\ar[r] & H_{2}(X,\Z)
}
$$}
and the Five lemma, we could obtain that $\Phi_{1,3}(U)$ is an isomorphism if $\Phi_{1,3}(X):L_1H_3(X)\to H_3(X,\Z)$ is. Therefore, $\Phi_{1,3}(X)$ is not an isomorphism.
\end{proof}

\begin{remark}

From Lemma \ref{Lemma4.23} and Example \ref{Exam4.26}, for each $k\geq 3$, one can  construct  projective varieties $X$ admitting $\C^*$-action with isolated fixed points
 such that $L_1H_k(X)\ncong H_k(X,\Z)$. Such a $X$ can be chosen as $X:=S\times S\times C^{k-3}$, where $C$ is the curve in Example \ref{exam4.2}.
 For $k=2$, a direct calculation shows that $L_1H_2(C\times C)\ncong H_2(C\times C,\Z)$.
  The detail  is left to the interested reader.
\end{remark}

\section{Applications to  Chow varieties}
\label{Section5}
 In this section, we shall first very briefly  review some known facts about Chow varieties, especially in algebraic and topological aspects and then give some new results.
Unless otherwise specified, Chow varieties defined over the complex numbers.

One of our purpose is to understand the algebraic and topological structure on the complex Chow variety $C_{p,d}(\P^n)_{\C})$  (or simply $C_{p,d}(\P^n)$ if there is no confusion)
 parameterizing effective $p$-cycles of
degree $d$ in the complex projective space $\P^n$.

In  degree 1 case, $C_{p,1}(\P^n)$  is exactly the Grassmannian of $(p+1)$-planes in $\C^{n+1}$, which is a space
of fundamental importance in  geometry and topology. In dimension 0 case, $C_{0,d}(\P^n)$ is the $d$-th symmetric product of $\P^n$, a ``correct" object to realize homology when $d$
tends to infinity.
 It is needless to explain here the importance of Chow varieties in algebraic cycles theory. Until recent years, it is surprising that
not many topological and algebraic invariants were known about $C_{p,d}(\P^n)$ for $d>1$.

\subsection{The origin of Chow variety}
Let $X\subset \P^n$ be a complex projective variety and let $C_{p,d}(X)\subset C_{p,d}(\P^n)$ be the subset containing those cycles $c=\sum a_i V_i\in C_{p,d}(\P^n)$ whose support
$\supp(c)=\cup V_i$ lies  in $X$, where $V_i$ is an irreducible projective variety of dimension $\dim V_i=p$, $a_i\in \Z^+$ and $\sum a_i=d$. It has been established by Chow and Van der Waerden in 1937 that each  $C_{p,d}(X)$ canonically carries the structure of a projective algebraic set (see \cite{Chow-Waerden}).
More intrinsically, the space of all effective $p$-cycles can be written as a countable disjoint union $\coprod_{\alpha\in H_{2p}(X,\Z)} C_{p,\alpha}(X)$, where each $C_{p,\alpha}(X)$
carries the structure of  a  projective algebraic set.

\subsection{The dimension and number of irreducible components}
In general, $C_{p,d}(\P^n)$ is not irreducible.  The simplest non-irreducible Chow varieties is $C_{1,3}(\P^3)$, which has two irreducible components.
Moreover, the different irreducible  components may have different dimension. Examples of Chow varieties including those parametrizing curves of low degrees (less than or equals to 4) in $\P^3$ can be
found in \cite{Gelfand-Kapranov-Zelevinsky}.

The exact number of irreducible components for $C_{p,d}(\P^n)$ is not known in general, even for $C_{1,d}(\P^3)$.
An upper bound of the number of irreducible components of $C_{p,d}(\P^n)$ was given by
$
N_{p,d,n}:=\bigg(^{nd+d}_{\quad n}\bigg)^{m_{p,d}},
$
where $m_{p,d}:=d\bigg(^{d+p-1}_{\quad p}\bigg)+\bigg(^{d+p-1}_{p-1}\bigg)$ (see Kollar \cite[Exer.3.28]{Kollar}).
We should mention that Kollar's book contains an excellent exposition on families of cycles over  arbitrary schemes.
Of course, this upper bound is usually much higher than the actual number of irreducible components for $C_{p,d}(\P^n)$
in many known cases. For example, there is exactly one component for $C_{0,d}(\P^n)$ for any $d$ and $n$.
For $d=1$ and arbitrary $n,p\geq 0$, $C_{p,1}(\P^n)$ is the Grassmannian parametering $(p+1)$-vector spaces in $\C^{n+1}$, which is irreducible.
For $d=2$ and arbitrary $n,p\geq 0$, there are at most two irreducible components for $C_{p,2}(\P^n)$.
By checking the possible genus of an irreducible curve with a given degree in $\P^3$ (see \cite[Ch. IV]{Hartshorne}), one can obtain that the irreducible components of $C_{1,d}(\P^3)$ are 1,2,4,8,14,27,46 corresponding to $d$ from 1 to 7. These numbers are really much smaller than the corresponding numbers $N_{p,d,n}$.

 The dimension of $C_{p,d}(\P^n)$ we mean the maximal
number of the dimension of its irreducible components.
Eisenbud and Harris in 1992 showed that the dimension of the space of effective  1-cycles of degree $d$ in $\P^n$  is
$$\dim C_{1,d}(\P^n)=\max\{2d(n-1),3(n-2)+d(d+3)/2 \}$$
(see \cite{Eisenbud-Harris}).

The dimension of $C_{p,d}(\P^n)$ was computed by Azcue in 1992 in his Ph.D. thesis under the direct of Harris (see \cite{Azcue}).
The explicit formula for  $\dim C_{p,d}(\P^n)$  can be found in a paper by Lehmann in 2017(see \cite{Lehmann}), that is,
$$\dim C_{p,d}(\P^n)=\max \Bigg\{ d(p+1)(n-p),\, \binom {d+p+1}{p+1} -1 + (p+2)(n-p-1) \Bigg\}.$$

\subsection{Homotopy and homology groups}
It is not hard to show that $C_{p,d}(\P^n)$ is connected as a topological space since every element $c$ is path-connected to $d\cdot L$, where $L$ is any fixed $p$-plane in $\P^n$.
By comparing connectedness of the morphism between  a variety and the fixed point set under the additive group  action,   Horrocks showed in 1969  that the algebraic fundamental group of
the Chow variety $C_{p,d}(\P^n)_{K}$  defined over an algebraically closed field $K$ is  trivial (see \cite{Horrocks}).
By using the similar method to complex varieties, A. Fujiki showed in 1995 that the topological fundamental group of  $C_{p,d}(\P^n)$ is trivial, i.e., $C_{p,d}(\P^n)$ is simply connected (see \cite{Fujiki}).

In a complete different way, Lawson in 1989 gave a very short proof of the simply connectedness of  $C_{p,d}(\P^n)$ by using Sard theorem for families (see \cite{Lawson1}). More important, in that paper, Lawson has established the Lawson homology theory and showed the famous Complex Suspension Theorem. The  author has observed that the methods in proving the Complex Suspension Theorem  can be used  to compute the higher homotopy group of $C_{p,d}(\P^n)$. The  author showed in 2010 that $\pi_2(C_{p,d}(\P^n))\cong \Z$ for all $d\geq 1$ and $0\leq p<n$. This statement $\pi_2(C_{p,d}(\P^n))\cong \Z$ was conjectured by Lawson in 1995 in \cite[p.141]{Lawson2}. For $p=n$, $C_{p,d}(\P^n)$ is a single point and so $\pi_2(C_{p,d}(\P^n))$ is trivial.  More results can be found in \cite{Hu3}
 on the stability of the homotopy group of $C_{p,d}(\P^n)$ when $p$ or $n$ increases.

For higher homotopy groups, a slightly weaker version of Lawson's open question is that
whether there is an isomorphism $\pi_k(C_{p,d}(\P^n))\cong \widetilde{H}_{k+2p}(\P^n,\Z)$ for $k\leq 2d$, where $\widetilde{H}_.(-,\Z)$ denotes the reduced singular homology with integer coefficients (see \cite[p.256]{Lawson1}). Lawson  showed  that
 there is a natural surjective map from $\pi_k(C_{p,d}(\P^n))$ to $\widetilde{H}_{k+2p}(\P^n,\Z)$. The  author showed in 2015 that the surjective map is actually an isomorphism.
Moreover, as its corollary, the homology group of $C_{p,d}(\P^n)$ has been computed up to $2d$ (see \cite{Hu4}).

\subsection{Euler characteristic}
By establishing a fixed point formula for compact complex spaces under a weakly analytic $S^1$-action, Lawson and Yau showed in 1987 that
the Euler characteristic $\chi(C_{p,d}(\P^n))$ of the complex Chow variety is given by a beautiful formula
$$\chi(C_{p,d}(\P^n))=\big(^{v_{p,n}+d-1}_{\quad\quad d}\big),$$ where $v_{p,n}=(^{n+1}_{p+1})$.

In 2013, the  author  gave a  direct and elementary proof of this formula (see \cite{Hu5}).  One of the main techniques  is ‘‘pulling of normal
cone" established  by Fulton, which was used by Lawson in proving his Complex Suspension Theorem (see \cite{Lawson1}).
The key observation  was that one can write $C_{p+1,d}(\P^{n+1})$ as disjoint union of quasi-projective varieties
$$C_{p+1,d}(\P^{n+1})_i=\coprod_{i=0}^d C_{p+1, i}(\P^n)\times T_{p+1,d-i}(\P^{n+1}),$$
where $T_{p+1,d-i}(\P^{n+1})$ is homotopic to $C_{p,d-i}(\P^{n})$ by the technique ‘‘pulling of normal
cone" . Hence one obtains  by the additive property of the Euler characteristic a recursive formula
$$\chi(C_{p+1, d}(\P^{n+1}))=\chi(C_{p, d}(\P^{n}))+\sum_{i=1}^d \chi(C_{p+1, i}(\P^n))\cdot\chi(C_{p,d-i}(\P^{n}))
$$ and complete the short proof of Lawson and Yau's formula.

The technique above are also able to use the compute the $l$-adic Euler-Poincar\'{e} characteristic of the Chow varieties $C_{p, d}(\P^{n})_K$ defined over
an algebraically closed field $K$. As an analogue in complex case, Friedlander showed in 1991 that there is an algebraic homotopy from  $T_{p+1,d-i}(\P^{n+1})$ to $C_{p,d-i}(\P^{n})$.
One got the generalization of Lawson-Yau's formula directly to Chow varieties  over an algebraically closed field $K$:
$$
\chi(C_{p,d}(\P^n)_K,l)=\big(^{v_{p,n}+d-1}_{\quad\quad d}\big), \quad\hbox{where $v_{p,n}=(^{n+1}_{p+1})$},
$$
where $\chi(X_K,l)$ denotes the $l$-adic Euler-Poincar\'{e} Characteristic of an algebraic variety $X_K$ over $K$.
The Euler Characteristic for the space of right-quaternionic cycles was also given with an explicit formula (see \cite{Hu5}).

 It seems that there is no way to compute the Euler characteristic $C_{p,\alpha}(X)$ for $X$ a generic projective variety. However, for
special varieties, such as toric varieties, Elizondo gave a beautiful formula for their  Euler characteristic in terms of the fans of the variety.

If one denotes the  \textbf{$p$-th Euler series} of a toric variety $X$ is defined by the following formal power series
$$
E_{p}(X):=\sum_{\alpha\in H_{2p}(X,\Z)}\chi(C_{p,\alpha}(X))\alpha.
$$

A toric variety  $X$ is a projective variety containing the algebraic
group $T=(\C^*)^{\times n}$ as a Zariski open
subset such that the action of $(\C^*)^{\times n}$  on itself extends to an action on $X$.
The action of $T$ on $X$ induces action on  $C_{p,\alpha}(X)$.

Denote by $V_1,...,V_N$ the $p$-dimensional invariant irreducible subvarieties of $X$. Let $e_{[V_i]}$ be the
characteristic function  of the subset $\{[V_i], i=1,2,..., N\}$ of the homology group $H_{2p}(X,\Z)$, where $[V]$ denotes its class
in $H_{2p}(X,\Z)$. Elizondo showed in 1994 that there is a beautiful formula for $E_{p}(X)$:
$$
E_{p}(X)=\prod_{1\leq i\leq N}\bigg(\frac{1}{1-e_{[V_i]}}\bigg).
$$

Elizondo and Lima-Filho showed 1998 taht The Euler-Chow series of the projectivization of the direct sum of two algebraic vector bundles  can be computed
in terms of that of the  projectivization of each of the vector bundles and their fiber product (see \cite{Elizondo-Lima}). More specifically,
let $E_1$ and $E_2$ be two algebraic vector bundle
over a projective variety $X$. Let $\P(E_1)$ (resp. $\P(E_2)$) be the projectivization of $E_1$ (resp. $E_2$).
Then the Euler-Chow series $E_p(\P(E_1\oplus E_2))$ can be computed in terms of that of $\P(E_1)$, $\P(E_2)$ and
$\P(E_1)\times_X \P(E_2)$, where the last one is the fiber product of $\P(E_1)$ and $\P(E_2)$ over $X$. This result  can be used to
compute the Chow series of Grassmannian and certain flag varieties.

\subsection{Virtual Betti and Hodge numbers }

For integers $n\geq p\geq 0$ and $d\geq 0$,
the  author showed in 2013 that
the  virtual Hodge $(r,s)$-number
of the Chow variety $C_{p,d}(\P^n)$ satisfies the following equations:
$$
\sum_{r-s=i}\tilde{h}^{r,s}(C_{p,d}(\P^n))=0
$$
for all $i\neq 0$,
$$
\sum_{r\geq0}\tilde{h}^{r,r}(C_{p,d}(\P^n))=\chi(C_{p,d}(\P^n)),
$$
$$\tilde{h}^{0,0}(C_{p,d}(\P^n))=1,$$
$$\tilde{h}^{r,0}(C_{p,d}(\P^n))=0,$$
and
$$\tilde{h}^{0,r}(C_{p,d}(\P^n))=0$$
for $r>0$ (see \cite{Hu}).

This also implies that  $\tilde{\beta}^{0}(C_{p,d}(\P^n))=1$ and
  $\tilde{\beta}^{1}(C_{p,d}(\P^n))=0$. It is worth
  to remark that for a complex singular projective variety $X$, $\tilde{\beta}^{0}(X)=1$ is independent of the
  connectedness of $X$, while $\tilde{\beta}^{1}(X)=0$ is independent of the simply connectedness of $X$.

Due to the lack understanding of the structure of $C_{p,d}(\P^n)$, we post the following  wild conjecture
on their virtual Hodge numbers and virtual Betti numbers.
\begin{conjecture}\label{Conj5.1}
$\tilde{h}^{r,s}(C_{p,d}(\P^n))=0$ for all $r\neq s$.
In particular, we conjecture that $\tilde{\beta}^{i}(C_{p,d}(\P^n))=0$ for $i$ odd.
\end{conjecture}

There are several examples supporting this conjecture. When $p=0$, $C_{p,d}(\P^n)=\sp^d(\P^n)$, its virtual Betti numbers and virtual Hodge numbers
have been computed in \cite{Cheah} and all their odd virtual Betti and virtual Hodge numbers vanish. When $p=n-1$, $C_{p,d}(\P^n)=C_{n-1,d}(\P^n)=\P^{(^{n+d}_{~d})-1}$ and
its  virtual Betti (resp. virtual Hodge numbers) are the same as its  usual Betti numbers (resp.usual Hodge numbers), which are zeroes.
  When $d=1$, $C_{p,d}(\P^n)$ is the Grassmannian  $G(p+1,\C^{n+1})$,  then one has  $\tilde{h}^{r,s}(G(p+1,\C^{n+1}))={h}^{r,s}(G(p+1,\C^{n+1}))=0$ for all $r\neq s$,
where ${h}^{r,s}(G(p+1,\C^{n+1}))$ denotes the Hodge $(r,s)$-number of  $G(p+1,\C^{n+1})$.

\begin{example}
For $d=2$ and all $p,n$, one also has ${h}^{r,s}(C_{p,2}(\P^n))=0$ for $r\neq s$ and
${\tilde{\beta}}^{2i-1}(C_{p,2}(\P^n))=0$ for  $i>0$.
\end{example}
\begin{proof}
Note that $C_{p,2}(\P^n)$ can be written as the union  $$C_{p,2}(\P^n)=\sp^2(G(p+1,\C^{n+1}))\cup Q_{p,n},$$
where $Q_{p,n}$  consists of effective irreducible $p$-cycles of degree $2$ in $\P^n$ and
$Q_{p,n}$ is a fiber bundle over the Grassmannian $G(p+2,n+1)$ with fiber the space $S$ of all smooth quadrics in
$\P^{p+1}$. Note that $S$ is isomorphic to $\P^{(^{p+3}_{~2})-1}-\sp^2(\P^{p+1})$ (see \cite{Hu3}).
Therefore,
{\small
$$
\begin{array}{ccl}
\widetilde{P}_{C_{p,2}(\P^n)}(t)&=&\widetilde{P}_{\sp^2(G(p+1,\C^{n+1}))}(t)+\widetilde{P}_{ Q_{p,n}}(t) \\
&=& \widetilde{P}_{\sp^2(G(p+1,\C^{n+1}))}(t)+\widetilde{P}_{G(p+2,n+1)}\cdot  \widetilde{P}_{\P^{(^{p+3}_{~2})-1}-\sp^2(\P^{p+1})}(t) \\
&=& \widetilde{P}_{\sp^2(G(p+1,\C^{n+1}))}(t)+\widetilde{P}_{G(p+2,n+1)}\cdot  (\widetilde{P}_{\P^{(^{p+3}_{~2})-1} }(t) -\widetilde{P}_{ \sp^2(\P^{p+1})}(t)).
\end{array}
$$ }
This implies that the odd betti numbers of $C_{p,2}(\P^n)$ are zeroes since those of Grassmannians and the symmetric product of Grassmannians
are zeroes.  Similar computations works for the virtual Hodge numbers.
\end{proof}

\subsection{Ruledness and Rationality of irreducible components}
Since $C_{p,d}(\P^n)$ admits a $\C$-action with an isolated fixed point (\cite{Horrocks}), each of its irreducible component is preserved under the action.
Hence each irreducible component of $C_{p,d}(\P^n)$ admits a  $\C$-action with an isolated fixed point.  From Lieberman's result (\cite[Th.1]{Lieberman1}), we obtain
that each component of $C_{p,d}(\P^n)$ is an \emph{ruled} variety.

In general, the  rationality of  irreducible components of $C_{p,d}(\P^n)$ is an open problem, which can be found in Shafarevich's book (see \cite{Shafarevich}).
As a remark, Shafarevich said  ``\emph{Whether every irreducible component of them is rational, in general,  is  `an apparently very difficult but very fundamental problem'}."

\begin{question}[Shafarevich]\label{ques5.3}
 Is each irreducible component
of $C_{p,d}(\P^n)$ is rational for all $0\leq p\leq n$ and $d\geq 1$?
\end{question}

Surely, $C_{p,1}(\P^n)$ is rational for all $0\leq p\leq n$ since $C_{p,1}(\P^n)$ is just the complex Grassmannian manfold $G(p+1,\C^{n+1})$, which is rational.
When $p=1,n=3$, the irreducible components of  $C_{1,d}(\P^3)$ have been shown to be rational for $d$ small (\cite{Shafarevich}).
However, even if  the proof of rationality for  $C_{0,d}(\P^n)$ is  nontrivial (see \cite[Ch.4,Th2.8]{Gelfand-Kapranov-Zelevinsky} and references cited there).

For $d=2$ and $0\leq p\leq n$, the explicit structure of each irreducible component has been studied in details in \cite{Hu3}. From that, one obtains that each  irreducible component is rational
since the symmetric products of complex Grassmannian manfolds are rational.

It is not hard to show that an irreducible component of the maximal dimension in $C_{p,d}(\P^n)$ is rational. This follows from  the fact that the symmetric product of a rational variety is  rational
and at least one irreducible component of the maximal dimension either consists of all $d$-tuples $p$-dimensional linear spaces in $\P^n$ or
irreducible $p$-dimensional hypersurfaces degree $d$ in $\P^{p+1}\subset \P^n$ (see \cite{Lehmann}).

However,
the answer to Question \ref{ques5.3} is negative, as explained in the following counterexample, which should be known earlier but
it cannot be found in the literature.
\begin{example}\label{Exam5.4}
Let $M_g$ ($g\geq 2$)be the moduli space of smooth complex algebraic curves of genus $g$.
Now we recall the  construction of  $M_g$ from the geometric invariant theory (cf. \cite{Harris-Morrison}).
Let $\mathcal{H}_{d,g,r}$ be the Hilbert scheme  of curves of degree $d$ and (arithmetic)genus $g$ in $\P^r$.
For any integer $n\geq 3$, a smooth curve $C$ can be embedded as a curve of degree $2(g-1)n$  in $\P^N$ by the complete linear series $|nK_C|$, where
$N=(2n-1)(g-1)-1$. Let us consider pairs $(C,\varphi:C\to \P^N)$, where $C$ is a curve and $\varphi:C\to \P^N$ is an $n$-canonical embedding. The family
of all such pairs corresponds to a locally closed subset $\mathcal{K}$ of the Hilbert scheme $\mathcal{H}_{d,g,N}$ of smooth curves of degree $d$ and genus $g$ in $\P^N$, where $d=2(g-1)$. The projective general linear
 group $PGL(N+1,\C)$ acts on $\mathcal{K}$ with quotient is $M_g$. The locally closed subset $\mathcal{K}$ is just a Zariski open set of an irreducible component of $C_{1,d}(\P^N)$.
 Therefore, there exists an irreducible component of $C_{1,d}(\P^N)$, denoted by $I_{1,d}(\P^N)$, and a dominant rational  map $I_{1,d}(\P^N)\dashrightarrow M_g$ for each $g\geq 2$.

Note that it was shown in \cite{Eisenbud-Harris2} and \cite{Harris-Mumford} that  $M_g$ is a quasi-projective variety of the general type for $g\geq 24$. This together with the
dominant rational map $I_{1,d}(\P^N)\dashrightarrow M_g$ implies that $I_{1,d}(\P^N)$ is not rational since a variety dominated by an rational variety is a unirational variety.
This completes the construction of the example.
\end{example}
One can  go further  to construct counterexamples to Shafarevich's question for cycles in arbitrary dimensions.

Fix a hyperplane $\P^n\subset \P^{n+1}$ and a point $P=[0:\cdots:0:1]\in \P^{n+1}-\P^n$.
 Let $V\subset \P^n$ be any closed algebraic subset. The algebraic suspension
of $V$ with vertex $P$ (i.e., cone over $P$) is the set
$$\Sigma_P V:=\cup\{ l~|~ l \hbox{ is a projective line through $P$ and intersects $V$}\}.
$$

Set $$T_{p+1,d}(\P^{n+1}):=\bigg\{ c=\sum n_iV_i\in C_{p+1,d}(\P^{n+1})| \dim(V_i\cap \P^n)=p, \forall i\bigg\}.$$
It has been shown in \cite{Lawson1} that $T_{p+1,d}(\P^{n+1})\subset C_{p+1,d}(\P^{n+1})$ is a Zariski open set and
there is a  continuous algebraic surjective map $T_{p+1,d}(\P^{n+1})\to C_{p,d}(\P^n)$  (cf. \cite{Friedlander1} for the case over arbitrary algebraically closed field).
A continuous algebraic map is a rational map which can be extended to a continuous map in the complex topology.
Hence, for each irreducible component $I_{p,d,n}$ of $C_{p,d}(\P^n)$, there exists an irreducible component
$J_{p+1,d,n+1}$ of $T_{p+1,d}(\P^{n+1})$ such that $J_{p+1,d,n+1}\to I_{p,d,n}$ is a   continuous algebraic surjective map. In particular, it is a  dominant rational  map.
Let $\overline{J}_{p+1,d,n+1}$ be the closure of $J_{p+1,d,n+1}$ in $C_{p+1,d}(\P^{n+1})$. Then we get a dominant rational  map
 $\overline{J}_{p+1,d,n+1}\dashrightarrow I_{p,d,n}$ from $J_{p+1,d,n+1}\to I_{p,d,n}$. Since $T_{p+1,d}(\P^{n+1})\subset C_{p+1,d}(\P^{n+1})$ is a Zariski open set, $\overline{J}_{p+1,d,n+1}$
 is an irreducible component $I_{p+1,d,n+1}$ of $C_{p+1,d}(\P^{n+1})$.
 So if $I_{p,d,n}\dashrightarrow M_g$ is a  dominant rational   map, then  $\overline{J}_{p+1,d,n+1}\dashrightarrow M_g$ is also a dominant rational map.
Therefore there is a   dominant rational  map is $I_{p+1,d,n+1}\dashrightarrow M_g$
  from the irreducible component $I_{p+1,d,n+1}$ of $C_{p+1,d,n+1}$ to the moduli space of curve of  genus $g$.
 From the construction of Example \ref{Exam5.4}, there exist $d,n$ such that $I_{1,d,n}\dashrightarrow M_g$ is a dominant rational map for $g\geq 2$.  Moreover, $M_g$ is of general type if $g\geq 24$
 by results in \cite{Eisenbud-Harris2} and \cite{Harris-Mumford}. Hence $I_{p+1,d,n+1}$ is not a rational variety since it dominates  a variety of general type.

In summary,  the above argument provides a proof to the following theorem  by induction.
\begin{theorem}\label{Thm5.5}
For any $p\geq 1$, there exists an irreducible component $I_{p,d,n}$ of $C_{p,d}(\P^n)$ such that $I_{p,d,n}$ is not rational if $d, n$  large.
\end{theorem}

\begin{remark}
The $I_{p,d,n}$ in Theorem \ref{Thm5.5} admits a $\C^*$-action with isolated fixed points but it is not rational.
\end{remark}

\subsection{Chow groups and Lawson homology }

By using the results in the sections above, we shall compute the Chow groups of 0-cycles  and Lawson homology of 1-cycles for Chow varieties $\Ch_0(C_{p,d}(\P^n))$.

We consider the action of $\C^*$ on $\P^{n}$ given by setting
$$\Phi_t([z_0:...:z_{n}])=[t_0z_0:...:t_nz_{n}],$$
where $t=(t_0:...:t_{n})\in (\C^*)^{n+1}$ and $[z_0:...:z_{n}]$ are homogeneous coordinates for $\P^{n+1}$.

This action on $\P^{n}$ induces an action of $(\C^*)^n$ on $C_{p,d}(\P^{n})$.
From the definition of the action $(\C^*)^n$ on $\P^{n}$, it is pretty clear that any irreducible subvariety $V$ of $\dim V=p$
is invariant under the action $(\C^*)^n$ if and only if  $V$ is spanned by $(p+1)$-coordinate points in $\P^n$ and hence the
fixed point set is finite.

\begin{proposition}
For all $d>0,0\leq p\leq n$, we have
$$\Ch_0(C_{p,d}(\P^n))\cong \Z.$$
\end{proposition}
\begin{proof}

Since $C_{p,d}(\P^n)$ admits a $(\C^*)^n$-action with isolated fixed points, one can get a $\C^*$-action
on  $C_{p,d}(\P^n)$ with isolated fixed points. Now by Proposition \ref{Prop4.19}, we get $\Ch_0(C_{p,d}(\P^n))\cong \Z$
since $C_{p,d}(\P^n)$ is connected.
\end{proof}

\begin{proposition}
For all $d>0,0\leq p\leq n$, we have
$$L_1H_k(C_{p,d}(\P^n))\cong H_k(C_{p,d}(\P^n),\Z)$$ for all $k\geq 2$. In particular, $L_1H_2(C_{p,d}(\P^n))\cong\Z$.
Equivalently, the homotopy groups of the space of 1-cycles  of the Chow variety $C_{p,d}(\P^n)$ coincides with the corresponding
singular homology groups with integer coefficients, i.e.,
$$\pi_{k-2}\cZ_1(C_{p,d}(\P^n))\cong H_k(C_{p,d}(\P^n),\Z)$$
for all $k\geq 2$.
\end{proposition}
\begin{proof}
By \cite{Horrocks}, we know $C_{p,d}(\P^n)$ admits an action of a solvable group $G$ with a single fixed point, where
$G=G_r\supset G_{r-1}\supset\cdots\supset G_1\supset G_0=\{0\}$ is a normal series  with quotients $G_i/G_{i-1}$ isomorphic to the additive group scheme $\C$.

By Proposition \ref{Prop2.4}, we can show that if $X$ admits an action of a solvable group $G$ with a single fixed pint, then
$L_1H_k(X)\cong H_k(X,\Z)$.  In fact, we have the following inclusion $X^{G}=X^{G_r}\subset X^{G_{r-1}}\subset\cdots\cdots X^{G_{2}}\subset X^{G_{1}} \subset X^{G_{0}}=X$.
Since $G_r/G_{r-1}\cong\C$ and $X^{G_{r}}$ is a single point, we get byBy Proposition \ref{Prop2.4} that $L_1H_k(X^{G_{r-1}})\cong H_k(X^{G_{r-1}},\Z)$ from the fact $L_1H_k(X^{G_{r}})\cong H_k(X^{G_{r}},\Z)$.
Since $G_i/G_{i-1}\cong\C$ for all $i\geq 1$  and by induction andBy Proposition \ref{Prop2.4}, we have   $$L_1H_k(X^{G_{0}})\cong H_k(X^{G_{0}},\Z),$$  that is, $L_1H_k(X)\cong H_k(X ,\Z)$.

By applying this to $X=C_{p,d}(\P^n)$, we have $L_1H_k(C_{p,d}(\P^n))\cong H_k(C_{p,d}(\P^n),\Z)$ for all  $k\geq 2$.
This completes the proof of the proposition.
\end{proof}

Similar to Conjecture \ref{Conj5.1}, we post another wild conjecture
on their Chow groups and Lawson homology groups.
\begin{conjecture}
For $d\geq 0$ and $0\leq p\leq n$, one has
$$\Ch_q(C_{p,d}(\P^n))\cong H_{2q}(C_{p,d}(\P^n),\Z)$$ for all $ q\geq 0$
and $L_qH_k(C_{p,d}(\P^n))\cong H_k(C_{p,d}(\P^n),\Z)$ for all $k\geq 2q\geq 0$.
\end{conjecture}

\begin{remark}
For $p=0$, $C_{p,d}(\P^n)\cong \sp^d(\P^n)$, one can show that these conjectures are true in rational coefficients.
For  $1\leq p\leq n-2$ and $d$ large, we have no idea to show or disprove $\Ch_q(C_{p,d}(\P^n))\cong H_{2q}(C_{p,d}(\P^n),\Z)$ even for $q=1$.
\end{remark}

\emph{Acknowledgements.}
The project was partially sponsored by  STF of Sichuan province, China(2015JQ0007) and NSFC(11771305).

\end{document}